\documentclass[reqno]{amsart}
\usepackage{amssymb,latexsym}
\usepackage{enumerate}
\usepackage{amsmath}
\usepackage{graphicx}
\usepackage{verbatim}
\usepackage{hyperref}
\usepackage[all]{xy}
\usepackage{color}

\newtheorem*{thm}{Theorem}

\newtheorem{theorem}{Theorem}[section]
\newtheorem{proposition}[theorem]{Proposition}
\newtheorem{lemma}[theorem]{Lemma}
\newtheorem{corollary}[theorem]{Corollary}

\theoremstyle{definition}
\newtheorem{definition}[theorem]{Definition}
\newtheorem{example}[theorem]{Example}
\newtheorem{remark}[theorem]{Remark}

\newtheorem{problem}[theorem]{Problem}

\newtheorem*{rem}{Remark}

\newcommand{\enproof}{\hspace*{\stretch{1}}\qedsymbol}  

\newcommand{\norm}[1]{\left\Vert#1\right\Vert}  

\newcommand{\Lspace}{\mathit{L}}
\newcommand{\Lone}{\Lspace^1}
\newcommand{\Linfty}{\Lspace^\infty}
\newcommand{\lspace}{\ell}
\newcommand{\lone}{\lspace^1}

\newcommand{\co}{\mathit{c}_{\,0}}

\newcommand{\C}{\mathcal{C}}
\newcommand{\UC}{\mathcal{UC}}
\newcommand{\set}[1]{\left\{#1\right\}}  

\newcommand{\Measures}{\mathcal{M}}
\newcommand{\haar}{\mathrm{m}}

\newcommand{\complexs}{\mathbb{C}}    

\newcommand{\reals}{\mathbb{R}}    
\newcommand{\naturals}{\mathbb{N}}    
\newcommand{\integers}{\mathbb{Z}}    
\newcommand{\rationals}{\mathbb{Q}}    

\newcommand{\tuple}[1]{\boldsymbol{#1}}

\newcommand{\Ccal}{\mathcal{C}}
\newcommand{\Dcal}{\mathcal{D}}

\newcommand{\Ecal}{\mathcal{E}}
\newcommand{\Fcal}{\mathcal{F}}

\newcommand{\Tcal}{\mathcal{T}}

\newcommand {\Lfrak}{{\mathfrak L}}

\newcommand{\abs}[1]{\left|#1\right|}

\newcommand{\cardinality}[1]{\left|#1\right|}

\newcommand{\dual}[1]{#1^\prime}
\newcommand{\bidual}[1]{#1^{\prime\prime}}
\newcommand{\predual}[1]{#1_*}

\newcommand{\cstar}{$C^*$}
\newcommand{\duality}[2]{\left\langle #1,#2 \right\rangle}

\newcommand{\linearmaps}{\mathcal{L}}
\newcommand{\operators}{\mathcal{B}}

\def\support{\mathop{\rm supp\,}}
\newcommand{\dd}{\,{\rm d}}

\begin{document}

\title[A combinatorial characterisation of amenability]
{A combinatorial characterisation of amenable locally compact groups}

\author{Hung\ Le\ Pham}
\address{School of Mathematics and Statistics\\
Victoria University of Wellington \\ Wellington 6140, New Zealand}
\email{hung.pham@vuw.ac.nz}


\begin{abstract}
We introduce a new combinatorial condition that characterises the amenability for locally compact groups. Our condition is weaker than the well-known F{\o}lner's conditions, and so is potentially useful as a criteria to show the amenability of specific locally compact groups. Our proof requires us to give a quantitative characterisation of (relatively) weakly compact subsets of  $\Lfrak^1$-spaces, and we do this through the introduction of a new notion of almost $(p,q)$-multi-boundedness for a subset of a Banach space that is intimately related to the well-known notion of the $(q,p)$-summing constants of an operator. As a side product, we also obtain a characterisation of weakly compact operators from $\Lfrak^\infty$-spaces in terms of their sequences of $(q,p)$-summing constants.
\end{abstract}




\maketitle

\section{Introduction} \label{Introduction}

\noindent Amenability  is an important property for locally compact groups that generalises both abelian and compactness, and has applications in many areas of mathematics; see, for examples, the books \cite{Greenleaf, Paterson, Pier, Runde}, and for its connection to Banach algebra theory, see  also \cite{HGD}. Thus characterisations of amenability are valuable tools, and among them, there is a class of related, combinatorial characterisations, which are commonly and collectively called \emph{F{\o}lner conditions} after E. F{\o}lner whose pioneering work in \cite{Folner} initiated their study in the case of discrete groups. 

Let $G$ be a locally compact group. The most well-known combinatorial condition that is equivalent to the amenability of $G$ (see \S \ref{Preliminaries} for the formal definition of amenability) is the following: 
\begin{itemize}
	\item[(F)]  For every $\varepsilon>0$ and every finite subset $F\subseteq G$, there exists a compact subset $C\subseteq G$ such that
	\[
		\haar(tC\setminus C)	<\varepsilon\,\haar(C)\quad(t\in F)\,.
	\]
\end{itemize}
This is easily seen to imply the amenability of $G$, as one can take a weak-$^*$ cluster point of $\chi_C/\haar(C)$ in $\dual{\Linfty(G)}=\bidual{\Lone(G)}$ as $F$ runs along the standard net of finite subsets of $G$ and $\varepsilon\searrow 0$; this remark seems to be first noted by Dixmier in \cite[3(b) of \S 4]{Dixmier}. It turns out that the converse is also true, so that (F) is equivalent to the amenability for $G$. This was proved by F{\o}lner \cite{Folner} for discrete $G$ and by Namioka \cite{Nam} for general $G$. There are in the literature several other combinatorial conditions of similar nature that are equivalent to the amenability for locally compact groups $G$. For example, the apparently strongest one is the following:
\begin{itemize}
	\item[(SF)] For every $\varepsilon>0$ and every compact subset $K\subseteq G$, there exists a compact subset $C\subseteq G$ such that
	\[
	\haar(KC\setminus C)<\varepsilon\,\haar(C)\,,
	\]
\end{itemize}
which was proved to be equivalent to the amenability of $G$ by Emerson and Greenleaf \cite{EG}. We refer the reader to \cite[Chapter 4]{Paterson}, \cite[\S 2.7]{Pier}, and \cite[\S 3.6]{Greenleaf} for more detailed discussion as well as for other related combinatorial conditions (we shall briefly discuss a couple of other conditions in \S \ref{Folner type condition section}). 

The known  F{\o}lner type conditions are all common in that they require the translations of the obtained set $C$ by elements of the given set $F$ vary little comparing to $C$. In this paper, we shall replace this requirement by a combinatorially weaker one where it is required instead that the net increase of the size of all the translations of $C$ by a sufficiently large subsets of $F$ is asymptotically smaller than what would be expected if $C$ and each of its translation have little in common. Our main theorem is the following.

\begin{thm}\label{main theorem}
Let $G$ be a locally compact group. Then the following are equivalent:
\begin{itemize} 
\item[\rm{(i)}] $G$ is amenable;
\item[\rm{(ii)}]  for every $\varepsilon>0$, there exists $n_\varepsilon\in\naturals$ such that for every finite subset $F \subseteq G$ there exists a compact subset $C \subseteq G$ with the property that 
\[
	{\haar(EC)}< \varepsilon\cardinality{E}{\haar(C)}\quad\quad(E\subseteq F \ \textrm{with}\ \cardinality{E}\ge n_\varepsilon)\,;
\]
\end{itemize}
\end{thm}

\begin{rem}
In the case where $G$ is discrete, the above inequality has the following suggestive form:
\[
	\cardinality{EC}< \varepsilon \cardinality{E}\cardinality{C}.
\]
\end{rem}

The part that will require proof is the implication that (ii) implies that $G$ is amenable. In fact, we shall even weaken condition (ii) further to require only that it holds for some fixed $\varepsilon\in (0,1)$; see Theorem \ref{Folner condition 1} for more details. Our proof will be mostly functional analytic in that, roughly speaking, we shall look for an invariant mean on $\Linfty(G)$. We shall first use our combinatorial condition to construct a mean $\Lambda$ on $\Linfty(G)$ that satisfies
\begin{align}\label{Introduction: eq 1}
	\lim_{n\to\infty}\frac{\norm{(s_1\cdot\Lambda,\ldots,s_n\cdot\Lambda)}_n^{(1,1)}}{n}=0\,
\end{align}
for every sequence $(s_n)_{n=1}^\infty$ in $G$; where $\norm{\cdot}^{(1,1)}_n$ is  a special norm,  introduced in \cite[\S 4.1]{DP08}, that can be defined on $E^n$ for any given Banach space $E$;  see Definition \ref{(p,q)-multinorms}.

If it were known instead that the sequence of $\norm{(s_1\cdot\Lambda,\ldots,s_n\cdot\Lambda)}_n^{(1,1)}$ ($n\in\naturals$) is bounded, then \cite[Corollary 5.8]{DDPR1} would imply that the set $\set{s\cdot\Lambda\colon s\in G}$ is relatively weakly compact, and so Ryll-Nardzewski's fixed point theorem (\cite{R-N}, see also \cite{NamAsplund} or \cite{Nam1983} for short proofs) would give us an invariant mean. Although, we have here a weaker condition \eqref{Introduction: eq 1}, it turns out that  the set $\set{s\cdot\Lambda\colon s\in G}$ is still relatively weakly compact as a result. In order to prove this, we shall need to do a finer analysis of the weakly compact subsets of $\Lone(\mu)$-spaces, which in fact will give us a quantitative characterisation of weak compactness for subsets of $\Lone(\mu)$-spaces (Theorem \ref{weak compactness in Lone when multi-norm grow slowly}) that is independent of the lattice structure of the $\Lone(\mu)$-spaces. 

This quantitative characterisation will be given in terms of  a new notion called \emph{almost $(p,q)$-multi-boundedness} that is defined for subsets of Banach spaces (Definition \ref{almost (p,q)-multi-boundedness}) and for $1\le p, q<\infty$. The notion of almost $(p,q)$-multi-boundedness is defined by one of several equivalent conditions that arise from the classical summing norms, and recasting one of these conditions differently, we obtain a quantitative characterisation of weak compactness for operators from $\C(K)$-spaces in terms of the asymptotical property of their sequences of $(p,q)$-summing constants. 

Let us recall that weakly compact operators and absolutely $(q,p)$-summing operators have been well-studied in the theory of Banach spaces. We refer the reader to the books \cite{DJT}, \cite{Ja}, \cite{Wo}, and references therein. For example, Chapter 15 of \cite{DJT} deals exclusively with weakly compact operators from the classical Banach spaces $\C(K)$ of continuous functions on compact spaces $K$.  When $p=q$, it is a consequence of the Pietsch factorisation theorem that every $p$-summing operator is weakly compact, but this is no longer true for $(q,p)$-summing operators when $p<q$ as shown in \cite{KP}. However, for operators from $\C(K)$, it was proved by Pisier \cite{Pisier} that every absolutely $(q,p)$-summing operator from $\C(K)$ is absolutely $r$-summing for any $r>q$, and thus is weakly compact. But since absolutely $r$-summing is a rather strong condition for an operator to hold, this leaves open the exact relation between the property of weak compactness for an operator from $\C(K)$ and its $(q,p)$-summing property. This will be clarified by our Theorem \ref{weak compactness of operators from abstract M spaces}. 

Although the concept of almost $(p,q)$-multi-boundedness has its classical root in the class of $p$-summing norms itself, it also has a more immediate origin from the theory of multi-norms, introduced in \cite{DP08}. This will be discussed in more details below. 

In fact, instead of $\Lone(\mu)$-spaces and of $\C(K)$-spaces, our results hold, respectively, for the more general classes of $\Lfrak^1$-spaces and of $\Lfrak^\infty$-spaces in the sense of Lindenstrauss and Pe{\l}czy\'{n}ski.

\section{Preliminaries} \label{Preliminaries}

\noindent Let $E$ be a Banach space; we shall denote its dual by $\dual{E}$, and its closed unit ball by $E_{[1]}$. In fact, for any subset $S$ of $E$ and any $r>0$, we shall write $S_{[r]}:=S\cap rE_{[1]}$.

Let $F$ be another Banach space. Then the linear space of all linear operators from $E$ to $F$ is denoted as $\linearmaps(E,F)$, whereas the Banach space of all bounded operators from $E$ to $F$ is denoted as $\operators(E,F)$.

Let $n\in\naturals$, and $1\le p<\infty$. Following the notation of \cite{Ja}, we define the \emph{weak $p$-summing norm} on $E^n$ by
\[
\mu_{p,n}(\tuple{x}):=\sup \set{\left(\sum_{i=1}^{n}\abs{\duality{x_{i}}{\lambda}}^{p}\right)^{1/p}: \lambda \in \dual{E}_{[1]}}\index{$\mu_{p,n}(\tuple{x})$}\,,
\]
where $\tuple{x}=(x_1,\ldots,x_n) \in E^n$. If $q\in [1,\infty)$ is another index, then  the \emph{$(q,p)$-summing constants} of an operator $T\in \linearmaps(E, F)$  are the numbers
\[
\pi_{q,p}^{(n)}(T) : = \sup\!\set{\left(\sum_{i=1}^n \norm{Tx_i}^{\,q}\right)^{1/q}\colon\ x_1,\dots, x_n\in E,\, \mu_{p,n}(x_1,\dots, x_n)\leq 1}.
\]
Obviously, the {$(q,p)$-summing constants} of $T$ form an increasing sequence. The operator $T$ is \emph{absolutely $(q,p)$-summing} if 
\[
	\pi_{q,p}(T):=\sup_{n\in\naturals} \pi_{q,p}^{(n)}(T)<\infty;
\]
if this is the case, we call $\pi_{q,p}(T)$ the \emph{$(q,p)$-summing norm} of $T$. When $q<p$, the only absolutely $(q,p)$-summing operator is the trivial one, so absolutely $(q,p)$-summing operators make sense only when $p\le q$. Also, when $q=p$, ``$(q,p)$-summing'' is abbreviated to ``\emph{$p$-summing}''.

Let $(\Omega,\mu)$ be a measure space, and $r\in[1,\infty]$. Then $\Lspace^r(\mu)=\Lspace^r(\dd\mu)=\Lspace^r(\Omega, \mu)$ denote the Banach space of (equivalence classes of) $r$-integrable functions [essentially bounded functions, when $r=\infty$] on $\Omega$ whose norm is denoted as $\norm{\,\cdot\,}_{\Lspace^r(\mu)}$. 

More generally, a Banach space $E$ is called an \emph{$\Lfrak^r$-space} if there is a $\lambda> 1$ such that, for any finite-dimensional subspace $M$ of $E$, there is a finite-dimensional subspace $N$ of $E$ containing $M$ and an isomorphism $T:N\to\lspace^r_{\dim N}$ such that $\norm{T}\norm{T^{-1}}\le \lambda$. This generalisation of $\Lspace^r(\mu)$-spaces was introduced in \cite[Definition 3.1]{LP}. 

Examples of $\Lfrak^r$-spaces include the spaces $\Lspace^r(\mu)$ above as well as their complemented subpaces. The class of $\Lfrak^\infty$-spaces also includes all $\C(K)$-spaces, their complemented subpaces, and their closed sublattices. 

In the case where $r=1$ or $r=\infty$, the defining condition of $\Lfrak^r$-spaces could be further relaxed, see for example \cite[Theorem 4.3]{LR}. However, the characterisation that is most important for us is that as shown in \cite[page 335]{LR}, for $r=1$ or $r=\infty$, a Banach space $E$ is an {$\Lfrak^r$-space} if and only if $\bidual{E}$ is a complemented subspace of $\Lspace^r(\mu)$ for some measure space $(\Omega,\mu)$. Finally, we note from \cite[Theorem III(a)]{LR} that a Banach space $E$ is an $\Lfrak^r$-space if and only if $\dual{E}$ is an $\Lfrak^{r'}$-space, where $r' \in [1,\infty]$ is conjugate to $r$.

Let $G$ be a locally compact group. We shall denote by $\haar=\haar_G$ the left Haar measure of $G$, and the integration of a function $f:G\to\complexs$ with respect to $\haar$ is denoted as $\int f(x)\dd x$. We also write $\abs{F}$  for the cardinality of a finite set $F$. Thus when $G$ is discrete, the standard convention will give $\haar(F)=\abs{F}$.

Recall that a locally compact group $G$ is \emph{amenable} if there exists a left-invariant mean $\Lambda$ on $\Linfty(G)$; that is a positive linear functional $\Lambda:\Linfty(G)\to\complexs$ such that
\[
	\Lambda(\tuple{1})=1\quad\text{and}\quad \delta_s\cdot\Lambda=\Lambda \qquad (s\in G)\,.
\]
Here $\delta_s$ denotes the point mass at $s\in G$, and $\delta_s\cdot \Lambda$, or more generally, $(\mu,\Lambda)\mapsto \mu\cdot \Lambda$ and $(\mu,\Lambda)\mapsto \Lambda\cdot \mu$ are the natural actions of the measure algebra $\Measures(G)$ on $\dual{\Linfty(G)}=\bidual{\Lone(G)}$ that arise through duality from the actions of $\Measures(G)$ on $\Lone(G)$ by convolution. More specifically, for every $\mu\in \Measures(G)$ and $\Lambda\in\dual{\Linfty(G)}$,
\begin{equation*}
\begin{aligned}\label{actions by measures}
	&\duality{\mu\cdot\Lambda}{\varphi}:=\duality{\Lambda}{\varphi\cdot\mu}\qquad(\varphi\in\Linfty(G))\;,\quad\text{and}\\
	\text{in turn}\quad &\duality{\varphi\cdot\mu}{f}:=\duality{\varphi}{\mu * f}\qquad(f\in\Lone(G))\;,
\end{aligned}
\end{equation*}
and similarly for the action on the other side. Recall also that if $G$ is amenable, then in fact there exists even a (two-sided) invariant mean on $\Linfty(G)$.

\section{Almost $(p,q)$-multi-boundedness}

\noindent Our starting point is the following theorem, which is an essential step in \cite{DDPR1} to prove that the injectivity of $\Lspace^p(G)$ (in the category of Banach left $\Lone(G)$-modules) implies the amenability of $G$.

\begin{theorem}\label{multi-bounded and weak compactness in Lone}
	\cite[Corollary 5.8]{DDPR1}
	Let $(\Omega,\mu)$ be a measure space, and suppose that $1\le p\le q<\infty$. Then every $(p,q)$-multi-bounded subset of $\Lone(\mu)$ is relatively weakly compact. 
\end{theorem}

Before explaining the terminology ``$(p,q)$-multi-bounded'', let us first recall the following definition from \cite[\S 4.1]{DP08}.

\begin{definition}\label{(p,q)-multinorms}  
Let $E$ be a Banach space, and take  $p,q\in[1,\infty)$. For each $n \in \naturals$
 and  $\tuple{x}=(x_1,\dots, x_n) \in E^{\,n}$, define
\begin{align*}
\norm{\tuple{x}}^{(p,q)}_n &:=
\sup \set{\left(\sum_{j=1}^n\abs{\duality{x_j}{\lambda_j}}^q\right)^{1/q}\colon\  \tuple{\lambda}=(\lambda_1,\dots,\lambda_n) \in (\dual{E})^n\ \textrm{with}\  \mu_{p,n}(\tuple{\lambda})\leq 1};
\end{align*} 
note that, when $E=\dual{\predual{E}}$ for some Banach space $\predual{E}$, we can let $\tuple{\lambda}$ runs over $(\predual{E})^n$ instead of $(\dual{E})^n$ in the above supremum. This will also be written as $\norm{\tuple{x}}^{(p,q)}_{E,n}$ when we want to stress its dependence on the space $E$, since it is not true in general that $\norm{\tuple{x}}^{(p,q)}_{E,n}=\norm{\tuple{x}}^{(p,q)}_{F,n}$ for $x\in F$ and $F$ a subspace of $E$.
\end{definition}

A subset $B$ of $E$ is called \emph{$(p,q)$-multi-bounded} \cite[Definition 5.1]{DDPR1} if
\begin{align}\label{(p,q)-multi-bounded}
\sup \set{ \norm{(x_1,\ldots, x_n)}^{(p,q)}_n: x_1,\ldots, x_n \in B,\, n \in \naturals}<\infty\,.
\end{align}
The idea is to use the multitude of norms $\norm{\cdot}^{(p,q)}_n$ ($n\in\naturals$) on tuples of elements of $B$ to give a stronger measurement of its boundedness than that is given by simply using the given norm of $E$.

As noted in \cite[Theorem 4.1]{DP08},  when $p\le q$, the sequence $(\norm{\cdot}^{(p,q)}_n: n \in \naturals)$ is actually a \emph{multi-norm} as introduced by Dales and Polyakov in \cite[Definition 2.1]{DP08}: that is, a sequence of norms $\norm{\cdot}_n$ on $E^n$ for $n \in \naturals$ that satisfies the following axioms:
\begin{itemize}
\item[(A1)] $\norm{(x_{\sigma(1)},\ldots,x_{\sigma(n)})}_{n}=\norm{(x_{1},\ldots,x_{n})}_{n}$ for each permutation $\sigma$;\smallskip

\item[(A2)] $\norm{(\alpha_1x_{1},\ldots,\alpha_nx_{n})}_{n}\leq \norm{(x_{1},\ldots,x_{n})}_{n} \quad (\alpha_1,\ldots,\alpha_n \in \complexs^n,\ \abs{\alpha_i}\le 1)$;\smallskip

\item[(A3)] $\norm{(x_1,\ldots,x_{n-1},0)}_{n}=\norm{(x_1,\ldots,x_{n-1})}_{n-1}$;\smallskip

\item[(A4)] $\norm{(x_1,\ldots,x_{n-2},x_{n-1},x_{n-1})}_{n}=\norm{(x_1,\ldots,x_{n-2},x_{n-1})}_{n-1}$;\smallskip
\end{itemize}
for all $n\geq 2$ and all $x_1,\ldots, x_n \in E$. The concept of $(p,q)$-multi-boundedness above is nothing but an instance of the more general notion of \emph{multi-boundedness} with respect to a general multi-norm given in \cite[Definition 6.3]{DP08}. For more details on this general theory, the reader is referred to the memoir \cite{DP08}. 

However, in this paper, we shall be contented with working with Definition \ref{(p,q)-multinorms} directly. The only general property of a multi-norm that we shall use is the following easy consequence of the axioms (A1)--(A4)
\begin{align}\label{a property of multi-norm}
	\norm{(y_1,\ldots,y_k)}_k\le\norm{(x_1,\ldots,x_n)}_n
\end{align}
whenever $\set{y_1,\ldots,y_k}\subseteq\set{x_1,\ldots,x_n}$, which could be directly checked for $(\norm{\cdot}^{(p,q)}_n: n \in \naturals)$. In fact, we see that $(\norm{\cdot}^{(p,q)}_n: n \in \naturals)$ satisfies \eqref{a property of multi-norm} even when $p>q$, even though it is no longer a multi-norm in that case. This is important for us, as we will allow the pair of $p$ and $q$ to range freely in $[1,\infty)$. On the other hand, we shall requires some specific properties of $(\norm{\cdot}^{(p,q)}_n: n \in \naturals)$ from \cite{DP08} and \cite{DDPR2}. 

We also need the following elementary property of $(\norm{\cdot}^{(p,q)}_n: n \in \naturals)$. Let $p,q,r\in [1,\infty)$ with $q\le r$. It is easy to see from Definition \ref{(p,q)-multinorms}  that, for every $x_1,\ldots, x_n\in E$,
\begin{align}\label{almost (p,q)-multi-bound is increasing} 
	\frac{\norm{(x_1,\ldots,x_n)}_n^{(p,q)}}{n^{1/q}}\le \frac{\norm{(x_1,\ldots,x_n)}_n^{(p,r)}}{n^{1/r}}.
\end{align}
On the other hand, if we fix $q$, then  
\begin{align}\label{almost (p,q)-multi-bound is increasing 2} 
	p\mapsto \norm{(x_1,\ldots,x_n)}_n^{(p,q)}\qquad\text{is an increasing function.}
\end{align}

Let us return our discussion to Theorem \ref{multi-bounded and weak compactness in Lone}. We first remark that the statement of this result is trivially true if we allow $p>q$ since in that case the only $(p,q)$-multi-bounded set (in any Banach space) is the singleton $\set{0}$. But even when $p\le q$, the result does not provide a full description of relatively weakly compact subsets of $\Lone(\mu)$ since its converse is not true. The following example provides  a subset of $\lone$ that is relatively compact, but is not $(p,q)$-multi-bounded for any $p,q\in[1,\infty)$.  

\begin{example}
For each $k\in\naturals$, define $g_k:=(1/\ln (k+1))\delta_k$, where $\set{\delta_1, \delta_2, \ldots}$ is the standard basis for $\lone$. Consider $B=\set{g_k:\ k\in\naturals}$. Then it is obvious that $g_k\to 0$ in norm in $\lone$ as $k\to\infty$. On the other hand,
\[
	\norm{(g_1,\ldots,g_{n})}^{(1,1)}_n=\sum_{k=2}^{n+1}\frac{1}{\ln k}\ge \int_2^{n+2}\frac{\dd t}{\ln t}\,.
\]
For each $1\le p\le q<\infty$, it follows from \eqref{almost (p,q)-multi-bound is increasing 2} and \eqref{almost (p,q)-multi-bound is increasing} that
\begin{align*}
	\norm{(g_1,\ldots,g_{n})}^{(p,q)}_n&\ge\norm{(g_1,\ldots,g_{n})}^{(1,q)}_n\\
	&\ge n^{1/q-1}\norm{(g_1,\ldots,g_{n})}^{(1,1)}_n\\
	&\ge \frac{\int_2^{n+2}{\dd t}/{\ln t}}{n^{1-1/q}}\to\infty\quad\textrm{as}\quad n\to\infty\,,
\end{align*}
where the limit is just a simple application of L'Hopital's rule. 
Thus $B$ is not $(p,q)$-multi-bounded for any $1\le p\le q<\infty$. \enproof
\end{example}

This shows that $(p,q)$-multi-boundedness is a genuinely stronger property than relative weak compactness in $\Lone(\mu)$. (In general, $(p,p)$-multi-bounded\-ness implies relative weak compactness in all Banach spaces, but, when $p<q$, $(p,q)$-multi-boundedness does not imply relative weak compactness in $\co$ \cite[Theorem 5.7 and Remark 5.9]{DDPR1}.) Thus, at least for the sake of studying the weak compactness property in $\Lone(\mu)$-spaces, we should first relax the property of $(p,q)$-multi-boundedness.

\medskip

\begin{lemma}\label{comparing formulae for almost (p,q)-multi-bound}
Let $E$ be a Banach space, and suppose that $1\le p, q<\infty$. Take a subset $B$ of $E$ and $c\ge 0$ such that
\[
	\inf_{n\in\naturals}\;\sup\set{\frac{\norm{(x_1,\ldots,x_n)}_n^{(p,q)}}{n^{1/q}}\colon\ x_1,\ldots,x_n\in B}>c\,.
\]
Then there exists a sequence $(x_n)$ in $B$ such that
\[
	\limsup_{n\to\infty}\frac{\norm{(x_1,\ldots,x_n)}_n^{(p,q)}}{n^{1/q}}>c\quad\textrm{and}\quad \inf_{n\in\naturals}\frac{\norm{(x_1,\ldots,x_n)}_n^{(p,q)}}{n^{1/q}}>\frac{c}{2^{1/q}}\,.
\]
\end{lemma}
\begin{proof}
The hypothesis implies that there exists $\varepsilon>0$ such that, for every $n\in\naturals$, there exist  $x_1,\ldots, x_n\in B$ with the property that 
\[
	\frac{\norm{(x_1,\ldots,x_n)}_n^{(p,q)}}{n^{1/q}}> c+\varepsilon\,.
\]
Set $n_0=0$. Suppose that $k\in\integers^+$, and assume inductively that we have already constructed $n_k\in\naturals$ and $x_1,\ldots,x_{n_k}\in B$. By the first sentence, we can find $n_{k+1}>n_k$ and $x_{n_k+1}$, $x_{n_k+2}$, \ldots, $x_{n_{k+1}}\in B$ such that
\[
	\frac{(n_{k+1}-n_k)^{1/q}}{n_{k+1}^{1/q}}>\frac{c+\varepsilon/2}{c+\varepsilon}\quad\textrm{and}\quad\frac{\norm{(x_{n_k+1},\ldots,x_{n_{k+1}})}_{n_{k+1}-n_k}^{(p,q)}}{(n_{k+1}-n_k)^{1/q}}> c+\varepsilon\,.
\]
In particular, there exist $\lambda_{n_k+1},\ldots,\lambda_{n_{k+1}}$ in $\dual{E}$ with $\mu_{p,n}(\lambda_{n_k+1},\ldots,\lambda_{n_{k+1}})\le 1$ such that
\[
	\frac{\left(\sum_{i=n_k+1}^{n_{k+1}}\abs{\duality{x_i}{\lambda_i}}^q\right)^{1/q}}{(n_{k+1}-n_k)^{1/q}}> c+\varepsilon\,.
\]
By reordering if necessary, we may suppose that the finite sequence $(\abs{\duality{x_i}{\lambda_i}})_{i=n_k+1}^{n_{k+1}}$ is decreasing. Hence we also have
\begin{align}\label{comparing formulae for almost (p,q)-multi-bound: eq 1}
	\frac{\norm{(x_{n_k+1},\ldots,x_{m})}_{m-n_k}^{(p,q)}}{(m-n_k)^{1/q}}> c+\varepsilon\quad(n_k+1\le m\le n_{k+1})\,.
\end{align}
This is continued inductively to obtain a strictly increasing sequence $(n_k)$ in $\naturals$ and a sequence $(x_n)$ in $B$. It remains to show that $(x_n)$ is the desired sequence.

First, we see that, for every $k\ge 0$
\begin{align*}
	\frac{\norm{(x_1,\ldots,x_{n_{k+1}})}_{n_{k+1}}^{(p,q)}}{n_{k+1}^{1/q}}\ge \frac{\norm{(x_{n_k+1},\ldots,x_{n_{k+1}})}_{n_{k+1}-n_k}^{(p,q)}}{(n_{k+1}-n_k)^{1/q}}\cdot\frac{(n_{k+1}-n_k)^{1/q}}{n_{k+1}^{1/q}}>c+\frac{\varepsilon}{2}\,.
\end{align*}
This proves the first stated inequality.

Second, let $m\in\naturals$ be arbitrary. Then there exists $k\in\integers^+$ such that $n_k+1\le m\le n_{k+1}$. If $k=0$, then, by \eqref{comparing formulae for almost (p,q)-multi-bound: eq 1}, we have
\[
	\frac{\norm{(x_1,\ldots,x_m)}_m^{(p,q)}}{m^{1/q}}>c+\varepsilon>\frac{c}{2^{1/q}}+\varepsilon\,.
\]
So, suppose that $k>0$. In this case, the previous paragraphs also imply that
\begin{align*}
	\frac{\norm{(x_1,\ldots,x_m)}_m^{(p,q)}}{m^{1/q}}&\ge \frac{\max\set{\norm{(x_1,\ldots,x_{n_k})}_{n_k}^{(p,q)}\,,\,\norm{(x_{n_k+1},\ldots,x_{m})}_{m-n_k}^{(p,q)}}}{m^{1/q}}\\
	&>\frac{\max\set{(c+\varepsilon/2)n_k^{1/q}\,,\,(c+\varepsilon)(m-n_k)^{1/q}}}{m^{1/q}}\ge \frac{c+\varepsilon/2}{2^{1/q}}\,.
\end{align*}
This proves the second stated inequality and the proof is completed.
\end{proof}

\smallskip

\begin{proposition}\label{equivalent formulae for almost (p,q)-multi-bound}
Let $E$ be a Banach space, and suppose that $1\le p, q<\infty$. Take a subset $B$ of $E$. Then the following formulae all define the same extended-real number: 
\begin{itemize}

\item[\rm{(i)}] $\displaystyle{\sup\set{\limsup_{n\to\infty}\frac{\norm{(x_1,\ldots,x_n)}_n^{(p,q)}}{n^{1/q}}\colon\ (x_n)\subseteq B}}\,;$\smallskip

\item[\rm{(ii)}] $\displaystyle{\inf_{n\in\naturals}\;\sup\set{\frac{\norm{(x_1,\ldots,x_n)}_n^{(p,q)}}{n^{1/q}}\colon\ x_1,\ldots,x_n\in B}}\,;$\smallskip

\item[\rm{(iii)}] $\displaystyle{\lim_{n\to\infty}\,\sup\set{\frac{\norm{(x_1,\ldots,x_n)}_n^{(p,q)}}{n^{1/q}}\colon\ x_1,\ldots,x_n\in B}}\,.$\smallskip
\end{itemize}  
\end{proposition}
\begin{proof}
Let us temporarily denote the extended-real numbers defined in (i) and (ii) as $c_1$ and $c_2$, respectively. By Lemma \ref{comparing formulae for almost (p,q)-multi-bound}, we see that $c_1\ge c_2$. We shall now prove that the limit in (iii) (exists) and is equal to $c_2$. For this, set
\[
	a_n:=\sup\set{\norm{(x_1,\ldots,x_n)}_n^{(p,q)}\colon\ x_1,\ldots,x_n\in B}\qquad(n\in\naturals).
\]
It is easy to see that $a_m^q+a_n^q\ge a_{m+n}^q$ for every $m,n\in\naturals$. Hence
\[
	\lim_{n\to\infty}\frac{a_n^q}{n}=\inf_{n\in\naturals}\frac{a_n^q}{n}=c_2^q\,,
\]
and so $\lim_{n\to\infty} a_n/n^{1/q}=c_2$. From this and the comparison of (i) and (iii), we immediately obtain that $c_1\le c_2$, and so $c_1= c_2$.

The proof is complete.
\end{proof}

\begin{remark}
We note in passing that if $B=E_{[1]}$, the closed unit ball of $E$, then
\[
	\sup\set{\frac{\norm{(x_1,\ldots,x_n)}_n^{(p,q)}}{n^{1/q}}\colon\ x_1,\ldots,x_n\in E_{[1]}}=\frac{\varphi^{(p,q)}_n(E)}{n^{1/q}}\qquad(n\in\naturals)\,
\]
where $(\varphi^{(p,q)}_n(E))_{n=1}^\infty$ is the rate of growth sequence corresponding to the $(p,q)$-multi-norm on $E$  (when $p<q$), defined in \cite[Definition 4.2]{DP08}.
\end{remark}

\begin{definition}\label{almost (p,q)-multi-bound}
Let $E$ be a Banach space, and suppose that $1\le p, q<\infty$. Let $B\subseteq E$. Then the number defined by any one of the three formulae in Proposition \ref{equivalent formulae for almost (p,q)-multi-bound} is denoted by $c_{p,q}(B)$. We also define the following related constant: 
\[
	d_{p,q}(B):=\sup\set{\inf_{n\in\naturals}\frac{\norm{(x_1,\ldots,x_n)}_n^{(p,q)}}{n^{1/q}}\colon\ (x_n)\subseteq B}\,.
\]
\end{definition}

Note that we always have
\[
	\frac{c_{p,q}(B)}{2^{1/p}}\le d_{p,q}(B)\le c_{p,q}(B)\,,
\]
where the first inequality follows from Lemma \ref{comparing formulae for almost (p,q)-multi-bound}.

\smallskip

If the set $B$ is unbounded, then we can choose $(x_n)\subseteq B$ such that $\norm{x_n}\ge n^2$ for each $n\in\naturals$. It then follows that
\[
	d_{p,q}(B)\ge \inf_{n\in\naturals}\frac{\norm{(x_1,\ldots,x_n)}_n^{(p,q)}}{n^{1/q}}\ge \inf_{n\in\naturals}\set{n^{2-1/q}\colon\ n\in\naturals}=\infty\,.
\]
Hence $c_{p,q}(B)=d_{p,q}(B)=\infty$.

On the other hand, if $B$ is bounded, then an easy calculation shows that
\[
	\frac{\norm{(x_1,\ldots,x_n)}^{(p,q)}}{n^{1/q}}\le b\qquad(x_1,\ldots, x_n\in B)\,,
\]
where $b$ is the bound of $B$. Thus both $c_{p,q}(B)$ and $d_{p,q}(B)$ are finite.

\medskip

The following remark follows from \eqref{almost (p,q)-multi-bound is increasing} and \eqref{almost (p,q)-multi-bound is increasing 2}.

\begin{remark}\label{comparing the almost (p,q)-multi-bounds}
Let $E$ be a Banach space, and let $B\subseteq E$. Then the functions $(p,q)\mapsto c_{p,q}(B)$ and $(p,q)\mapsto d_{p,q}(B)$ are increasing.
\end{remark}

\medskip

\begin{definition}\label{almost (p,q)-multi-boundedness}
Let $E$ be a Banach space, and suppose that $1\le p,  q<\infty$. A subset $B$ of $E$ is \emph{almost $(p,q)$-multi-bounded} if $c_{p,q}(B)=0$, or, equivalently, if $d_{p,q}(B)=0$.
\end{definition}

The following is immediate from the previous discussion, and summarises the main objective of this section.

\begin{theorem}
A subset $B$ of a Banach space $E$ is almost $(p,q)$-multi-bounded if and only if either one of the following four equivalent conditions are satisfied:
\begin{itemize}
\item[\rm{(i)}] for every sequence $(x_n)$ in $B$
\[
	\inf_{n\in\naturals}\frac{\norm{(x_1,\ldots,x_n)}_n^{(p,q)}}{n^{1/q}}=0\,;
\]
\item[\rm{(ii)}] for every sequence $(x_n)$ in $B$
\[
	\lim_{n\to\infty}\frac{\norm{(x_1,\ldots,x_n)}_n^{(p,q)}}{n^{1/q}}=0\,;
\]
\item[\rm{(iii)}] for every $\varepsilon>0$, there exists $n\in\naturals$ such that 
\[
	\frac{\norm{(x_1,\ldots,x_n)}_n^{(p,q)}}{n^{1/q}}< \varepsilon\quad (x_1,\ldots,x_n\in B)\,;
\]
\item[\rm{(iv)}] for every $\varepsilon>0$, there exists $n\in\naturals$ such that 
\[
	\frac{\norm{(x_1,\ldots,x_m)}_m^{(p,q)}}{m^{1/q}}< \varepsilon\quad (x_1,\ldots,x_m\in B)\,
\]
for each $m\ge n$. \enproof
\end{itemize}  
\end{theorem}
Note that condition (i) is apparently the weakest one among the four, while (iv) the strongest.

Since the sequence $(\norm{\cdot}^{(p,q)}_n\colon n\in\naturals)$ does not behave very well when passing to a subspace, we cannot expect that for, a subset of a normed subspace, being almost $(p,q)$-multi-bounded  is the same whether with respect to the subspace or the ambient space. However, under the following two special cases, this turns out to be true.

\begin{lemma}\label{passing to bidual and to complemented subspace}
Let $E$ be a Banach space, and let $B$ be a subset of $E$. Suppose that $1\le p, q<\infty$.
\begin{enumerate}
	\item $B$ is almost $(p,q)$-multi-bounded in $E$ if and only if it is almost $(p,q)$-multi-bounded in $\bidual{E}$.
	\item Suppose that $E$ is a complemented subspace of another Banach space $F$. Then  $B$ is almost $(p,q)$-multi-bounded in $E$ if and only if it is almost $(p,q)$-multi-bounded in $F$.
\end{enumerate}
\end{lemma}
\begin{proof}
(i) This follows from \cite[Corollary 4.14]{DP08}; in fact, it was proved there that $((\bidual{E})^n,\norm{\cdot}^{(p,q)}_n)$ is precisely the bidual of $(E^n,\norm{\cdot}^{(p,q)}_n)$. 

(ii) It is easy to see that, for any two Banach spaces $X$ and $Y$, and any $T\in\operators(X,Y)$, 
\[
	\norm{(Tx_1,\ldots,Tx_n)}^{(p,q)}_n\le \norm{T}\cdot \norm{(x_1,\ldots,x_n)}^{(p,q)}_n\qquad(x_1,\ldots,x_n\in X\,,\ n\in\naturals)\,.
\]
Now, let $P:F\to E$ be a projection. Then the above implies that
\[
	\norm{(x_1,\ldots,x_n)}^{(p,q)}_{F,n}\le \norm{(x_1,\ldots,x_n)}^{(p,q)}_{E,n}\le \norm{P}\cdot \norm{(x_1,\ldots,x_n)}^{(p,q)}_{F,n}\,,
\]
for every $x_1,\ldots,x_n\in E$ and every $n\in\naturals$. The statement (ii) then follows.
\end{proof}

\section{Almost $(p,q)$-multi-boundedness vs. weak compactness}

\noindent In this section, we shall apply the new concept introduced in the previous section to give quantitative characterisations of relatively weakly compact subsets of $\Lfrak^1$-spaces and of weakly compact operators from $\Lfrak^\infty$-spaces. We first consider the $\Lone(\mu)$-spaces, since in these we shall be able to connect the notion of almost $(p,q)$-multi-boundedness to the more classical characterisations of relatively weakly compact subsets of $\Lone(\mu)$ in terms of the integration structure of $\mu$. 

The classical characterisations referred above is the well-known theorem of Dunford and Pettis ({\it cf.} \cite[Theorem III.C.12]{Wo}). It states that the relative weak compactness of a bounded subset $B$ of $\Lone(\mu)$ is equivalent to each of the following conditions:
\begin{itemize}
\item[(i)] for every sequence $(X_n)$ of disjoint measurable subsets of $\Omega$
\[
	\inf_{n\in\naturals}\ \sup_{f\in B}\  \int_{X_n}\abs{f}\dd\mu=0\,;
\]
\item[(ii)] for every $\varepsilon>0$, there exists $n\in\naturals$ with the property that, for every pairwise disjoint family $\set{X_1,\ldots, X_n}$ of measurable subsets of $\Omega$, there exists $k$ such that
\[
	\sup\set{\int_{X_k}\abs{f}\dd\mu\colon f\in B}\le \varepsilon\,;
\]
\end{itemize}  
these correspond to clauses (c) and (e) of \cite[Theorem III.C.12]{Wo} (which is stated in the contrapositive form), respectively. Note that although \cite[Theorem III.C.12]{Wo} is stated and proved for finite measure spaces, the equivalence of relative weak compactness in $\Lone(\mu)$ and statements (i) and (ii) holds for any measure space $(\Omega,\mu)$; the case of a general measure can be reduced to that of a finite one by, for example, an use of the Eberlein--\v{S}mulian theorem. In the case of a finite measure $\mu$, the relative weak compactness of $B$ is also equivalent to its \emph{uniform integrability}, stated as follows:
\begin{itemize}
	\item[(iii)] for every $\varepsilon>0$, there exists $\delta>0$ such that
	\[
	\sup\set{\int_{X}\abs{f}\dd\mu\colon f\in B,\ \mu(X)<\delta}\le \varepsilon\,;
	\]
\end{itemize}  
this corresponds to clause (b) of \cite[Theorem III.C.12]{Wo}, and is obviously formally stronger than (i) and (ii) (when $\mu$ is finite).

We shall actually need the following strengthening of condition (ii) above. 

\begin{lemma}\label{weak compactness in Lone and lp bounded tuples in Linfty}
Suppose that $B$ is a relatively weakly compact subset of $\Lone(\mu)$. Let $p\in [1,\infty)$. Then, for every $\varepsilon>0$, there exists $n\in\naturals$ such that, for every  $\varphi_1,\ldots,\varphi_n\in \Linfty(\mu)$ with $\mu_{p,n}(\varphi_1,\ldots,\varphi_n)\le 1$ we have
\[
	\sup\set{\int\abs{f\varphi_k}\dd\mu\colon f\in B}\le \varepsilon\quad\text{for some $1\le k\le n$}\,.
\]
\end{lemma}
\begin{proof}
Remark that $\Linfty(\mu)$ is not always $\dual{\Lone(\mu)}$ via the natural identification, but for every $n\in\naturals$ and every $\varphi_1,\ldots,\varphi_n\in \Linfty(\mu)$
\begin{align}\label{weak compactness in Lone and lp bounded tuples in Linfty: eq -2}
	\mu_{p,n}(\varphi_1,\ldots,\varphi_n)\le 1\qquad\text{if and only if}\qquad \sum_{i=1}^n\abs{\varphi_i}^p\le 1\ \text{in $\Linfty(\mu)$}\,,
\end{align}
since $\Linfty(\mu)$ is an $\C(K)$-space. Let us first consider two special cases of the lemma.

\noindent \emph{Case 1: if $\mu$ is finite.} We may and shall suppose that $\mu$ is a probability measure. Denote by $b$ the bound of $B$. Since $B$ must now be uniformly integrable, we can find $\alpha>0$ such that
\begin{align}\label{weak compactness in Lone and lp bounded tuples in Linfty: eq -1}
\sup\set{\int_{X}\abs{f}\dd\mu\colon f\in B,\ \mu(X)\le\frac{b}{\alpha}}\le \frac{\varepsilon}{2}\,.
\end{align}
Then choose $n\in\naturals$ such that
\begin{align}\label{weak compactness in Lone and lp bounded tuples in Linfty: eq 0}
\frac{\alpha}{n^{1/p}}\le\frac{\varepsilon}{2}\,.
\end{align}
Take $\varphi_1,\ldots,\varphi_n\in \Linfty(\mu)$ with $\mu_{p,n}(\varphi_1,\ldots,\varphi_n)\le 1$ arbitrarily. Then by \eqref{weak compactness in Lone and lp bounded tuples in Linfty: eq -2}, there must exist an $k\in\set{1,\ldots,n}$ such that
\[
\frac{1}{n}\ge\int\abs{\varphi_k}^p\dd\mu\ge \left(\int\abs{\varphi_k}\dd\mu\right)^p\,.
\]
This and \eqref{weak compactness in Lone and lp bounded tuples in Linfty: eq 0} then imply that for each measurable function $f$
\begin{align*}
\int_{\set{\abs{f}\le \alpha}}\abs{f\varphi_k}\dd\mu\le\frac{\alpha}{n^{1/p}}\le\frac{\varepsilon}{2}\,.
\end{align*}
On the other hand, we see from \eqref{weak compactness in Lone and lp bounded tuples in Linfty: eq -1} that for $f\in B$
\begin{align*}
\int_{\set{\abs{f}> \alpha}}\abs{f\varphi_k}\dd\mu\le \int_{\set{\abs{f}>\alpha}}\abs{f}\dd\mu\le\frac{\varepsilon}{2}\,.
\end{align*}
These two inequalities together then prove the conclusion in the case where $\mu$ is finite.

\noindent \emph{Case 2: if $B$ is countable.} Says $B=\set{f_n\colon n\in\naturals}$. Set $w_0=\sum_{n=1}^\infty f_n/2^n$. Then $w_0\in \Lone(\mu)$ because $B$ is necessarily bounded in $\Lone(\mu)$. Let $\mu_0$ be a measure on $\Omega$, with the same $\sigma$-algebra as $\mu$, such that $\dd \mu_0=w_0\dd\mu$, and let $E$ be the image of the isometric mapping
\[
	f\mapsto fw_0,\ \Lone(\mu_0)\to \Lone(\mu)\,.
\]
Then $B\subseteq E$. Let $B_0$ be the subset of $\Lone(\mu_0)$ whose image is $B$. Then $B_0$ is relatively weakly compact in $\Lone(\mu_0)$. Since $\mu_0$ is a finite measure, \emph{Case 1} shows that $B_0$ satisfies the conclusion of the lemma with $\mu_0$ replacing $\mu$. Take $\varepsilon>0$, and let $n$ be as specified in the conclusion of the lemma for $B_0\subseteq \Lone(\mu_0)$. 

Take $\varphi_1,\ldots,\varphi_n\in \Linfty(\mu)$ with $\mu_{p,n}(\varphi_1,\ldots,\varphi_n)\le 1$. Then, since $\mu_0\ll \mu$, each $\varphi_j$ well-defines an element of $\Linfty(\mu_0)$, also denoted by $\varphi_j$. By \eqref{weak compactness in Lone and lp bounded tuples in Linfty: eq -2}, we see that $\mu_{p,n}(\varphi_1,\ldots,\varphi_n)\le 1$ in $\Linfty(\mu_0)^n$. Hence, for some $k\in\set{1,\ldots,n}$, we have
\begin{align*}
\varepsilon\ge \sup\set{\int\abs{f\varphi_k}\dd\mu_0\colon f\in B_0}=\sup\set{\int\abs{f\varphi_k}\dd\mu\colon f\in B}\,.
\end{align*}
This shows that the conclusion of the lemma holds for $B$ when it is countable.

\emph{In general:} Assume towards a contradiction that $B$ were a relatively weakly compact subset of $\Lone(\mu)$ for which the conclusion of the lemma failed. That would mean there exists an $\varepsilon_0>0$ such that for every $n\in\naturals$ there exist $\varphi_{n,1},\ldots,\varphi_{n,n}\in \Linfty(\mu)$ with $\mu_{p,n}(\varphi_{n,1},\ldots,\varphi_{n,n})\le 1$ and $f_{n,1},\ldots, f_{n,n}\in B$ such that
\[
	\int\abs{f_{n,k}\varphi_{n,k}}\dd\mu>\varepsilon\quad\text{for all $1\le k\le n$}.
\]
Then $\set{f_{n,k}\colon\ n\in\naturals, 1\le k\le n}$ is a countable relatively weakly compact subset of $\Lone(\mu)$ that fails the conclusion of the lemma. This contradicts \emph{Case 2}.
\end{proof}

We are now ready to state our characterisation of relatively weakly compact subsets of $\Lone(\mu)$, improving Theorem \ref{multi-bounded and weak compactness in Lone}. 

\begin{theorem}\label{weak compactness in Lone when multi-norm grow slowly}
Let $(\Omega,\mu)$ be a measure space, and let $p,q\in [1,\infty)$. Then a subset of $\Lone(\mu)$ is relatively weakly compact if and only if it is almost $(p,q)$-multi-bounded. 
\end{theorem}
\begin{proof}
Since the statement is about the norm structure of $\Lone(\mu)$, the actual measure space $(\Omega,\mu)$ that implements $\Lone(\mu)$ is not important. Thus by the Kakutani representation of abstract $\Lspace$-spaces, we may and shall suppose that $(\Omega,\mu)$ is chosen in such a way that $\dual{\Lone(\mu)}=\Linfty(\mu)$ via the canonical duality (see the proofs of Kakutani's theorem in \cite{Kakutani} or in \cite[Theorem 1.b.2]{LT} for this extra information).

Suppose that $B$ is a subset of $\Lone(\mu)$ that is not relatively weakly compact. If $B$ is not bounded, then $c_{p,q}(B)=\infty$, so that $B$ is not almost $(p,q)$-multi-bounded. Suppose now that $B$ is bounded. Then, by the theorem of Dunford and Pettis discussed above, there exist a sequence $(X_n)$ of pairwise disjoint measurable subsets of $\Omega$ and an $\varepsilon>0$ such that
\[
	\sup\set{\int_{X_{n}}\abs{f}\dd\mu\colon f\in B}>\varepsilon\quad\quad(n\in\naturals)\,.
\]
We then choose $f_n\in B$ and $\varphi_n\in\Linfty(\mu)_{[1]}$ for each $n\in\naturals$ such that $\support \varphi_n\subseteq X_{n}$ and such that $\duality{f_n}{\varphi_n}>\varepsilon$. Since $\varphi_j\in\Linfty(\mu)_{[1]}$ have disjoint supports, it is easily seen that $\mu_{p,n}(\varphi_1,\ldots,\varphi_n)\le 1$, and so
\[
	\norm{(f_1,\ldots,f_n)}_n^{(p,q)}>n^{1/q}\varepsilon\quad(n\in\naturals).
\]
This implies that $B$ is not almost $(p,q)$-multi-bounded. 

\smallskip

Conversely suppose that $B$ is relatively weakly compact. We need to show that $B$ is almost $(p,q)$-multi-bounded. Assume toward a contradiction that there exist a sequence $(f_n)$ in $B$ and an $\varepsilon>0$ such that 
\begin{align}\label{weak compactness in Lone when multi-norm grow slowly: eq 1}
	\frac{\norm{(f_1,\ldots,f_n)}_n^{(p,q)}}{n^{1/q}}> \varepsilon\quad(n\in\naturals)\,.
\end{align}
Fix $\delta\in(0,\varepsilon)$. By Lemma \ref{weak compactness in Lone and lp bounded tuples in Linfty}, there exists $n_0\in\naturals$ such that, for every  $\varphi_1,\ldots,\varphi_{n_0}\in \Linfty(\mu)$ with $\mu_{p,n_0}(\varphi_1,\ldots,\varphi_{n_0})\le 1$, 
\begin{align}\label{weak compactness in Lone when multi-norm grow slowly: eq 2}
\sup\set{\int\abs{f\varphi_k}\dd\mu\colon f\in B}\le \delta\quad\text{for some $1\le k\le n_0$}\,.
\end{align}
Let $n> n_0$. By \eqref{weak compactness in Lone when multi-norm grow slowly: eq 1}, there exist $\psi_1,\ldots,\psi_n\in \Linfty(\mu)$ such that $\mu_{p,n}(\psi_1,\ldots,\psi_n)\le 1$ and such that
\[
	\varepsilon n^{1/q}<\left(\sum_{k=1}^n\abs{\duality{f_k}{\psi_k}}^q\right)^{1/q}\,.
\]
By \eqref{weak compactness in Lone when multi-norm grow slowly: eq 2}, there must be $n-n_0$ indices $k$ in $\set{1,\ldots,n}$ for which
\[
	\int\abs{f_k\psi_k}\dd\mu\le \delta.
\]
It follows that
\begin{align*}
	\varepsilon^qn<\sum_{k=1}^n\abs{\int f_k\psi_k\dd\mu}^q\le (n-n_0)\delta^q+n_0\,b^q,
\end{align*}
where $b$ denotes the bound of $B$. Thus 
\[
	n<\frac{n_0(b^q-\delta^q)}{\varepsilon^q-\delta^q}\,,
\]
contradicting the fact that $n$ can be arbitrarily large. This proves that $B$ is almost $(p,q)$-multi-bounded. 

This completes the proof.
\end{proof}

In the case where $p=q=1$, by \cite[(4.12) and Theorem 4.26]{DP08}, the norm $\norm{\cdot}_n^{(1,1)}$ on $\Lone(\mu)^n$ has the following simple form 
\begin{align*}\label{max, [1,1], lattice multi-norm}
	\norm{(f_1,\ldots, f_n)}_n^{(1,1)}=\norm{\abs{f_1}\vee\cdots\vee\abs{f_n}}_{\Lone(\mu)}\quad(f_1,\ldots,f_n\in\Lone(\mu))\,;
\end{align*}
where $\vee$ is the joint operation of the lattice $\Lone_\reals(\mu)$. Thus we have the following consequence of the above theorem.

\begin{corollary}\label{weak compactness in Lone when max multi-norm grows slowly}
Let $(\Omega,\mu)$ be a measure space, and let $B$ be a subset of $\Lone(\mu)$. Then the following are equivalent:
\begin{itemize}
\item[\rm{(i)}] the set $B$ is relatively weakly compact;
\item[\rm{(ii)}]  for every sequence $(f_n)$ in $B$
\[
	\inf_{n\in\naturals}\frac{\norm{\abs{f_1}\vee\cdots\vee\abs{f_n}}_{\Lone(\mu)}}{n}=0\,;
\]
\item[\rm{(iii)}] for every sequence $(f_n)$ in $B$
\[
	\lim_{n\to\infty}\frac{\norm{\abs{f_1}\vee\cdots\vee\abs{f_n}}_{\Lone(\mu)}}{n}=0\,;
\]
\item[\rm{(iv)}] for every $\varepsilon>0$, there exists $n\in\naturals$ such that, for every $m\ge n$ and $f_1,\ldots,f_m\in B$, we have

\medskip

\indent\indent\indent \indent\indent\indent\indent\indent $\displaystyle{\norm{\abs{f_1}\vee\cdots\vee\abs{f_m}}_{\Lone(\mu)}\le \varepsilon m\,.}$ \enproof
\end{itemize}  
\end{corollary}

\begin{remark}
Since the statement of Corollary \ref{weak compactness in Lone when max multi-norm grows slowly} involves only the norm and the lattice structures of $\Lone(\mu)$, it holds for the Banach spaces of complex regular Borel measures on locally compact spaces as well.
\end{remark}

Combining the above theorem with Lemma \ref{passing to bidual and to complemented subspace} and \cite{LR}, we obtain the following generalisation. 

\begin{theorem}\label{weak compactness in L1 spaces}
Let $E$ be an $\Lfrak^1$-space, and let $p, q\in[1,\infty)$. Then a subset of $E$ is relatively weakly compact if and only if it is almost $(p,q)$-multi-bounded. 
\end{theorem}
\begin{proof}
Let $B$ be a subset of $E$. Let $(\Omega,\mu)$ be a measure space such that $\bidual{E}$ is (linear homeomorphic to) a complemented subspace of $\Lone(\mu)$; this is shown in \cite[page 335]{LR} to be possible. Then $B$ is almost $(p,q)$-multi-bounded in $E$ if and only if it is almost $(p,q)$-multi-bounded in $\Lone(\mu)$. Obviously, $B$ is relatively weakly compact in $E$ if and only if it is relatively weakly compact in $\Lone(\mu)$. The result then follows from the previous theorem. 
\end{proof}

The following examples show that almost $(p,q)$-multi-boundedness neither implies nor is implied by relative weak compactness in general Banach spaces.

\begin{example}
\begin{enumerate}
\item[{\rm (i)}] Consider $\co$, and denote by $(\delta_n)$ its standard basis sequence. In \cite[Remark 5.9]{DDPR1}, based on a result of Kwapie\'{n} and Pe{\l}czy\'{n}ski \cite{KP}, it is shown that the set $B:=\set{\sum_{i=1}^n\delta_i:\ n\in\naturals}$ is $(p,q)$-multi-bounded (and hence almost $(p,q)$-multi-bounded) for every $p,q$ with $1\le p<q<\infty$, but it is not relatively weakly compact. By Remark \ref{comparing the almost (p,q)-multi-bounds}, it also follows that $B$ is almost $(p,q)$-multi-bounded for every $p,q\in [1,\infty)$. In fact, it follows from \cite[Theorem 3.9]{DDPR2} as in the proof of Theorem \ref{weak compactness in L1 spaces} and Remark \ref{comparing the almost (p,q)-multi-bounds} that 
\begin{itemize}
	\item[] if $E$ is any infinite-dimensional $\Lfrak^\infty$-space,  then $E_{[1]}$ is not weakly compact but is almost $(p,q)$-multi-bounded for every $p, q\in [1,\infty)$ with $p<2$.
\end{itemize}

\item[{\rm (ii)}] Let $E$ be an infinite-dimensional reflexive Banach space. Then $E_{[1]}$ is weakly compact. However, it follows from \cite[Proposition 2.11(ii)]{DDPR2} and Remark \ref{comparing the almost (p,q)-multi-bounds} that $c_{p,q}(E_{[1]})=1$, so that $E_{[1]}$ is not  almost $(p,q)$-multi-bounded whenever $p, q\in [2,\infty)$.

\item[{\rm (iii)}] Let $r>1$. Then $B:=(\lspace^r)_{[1]}$ is weakly compact. However, it follows from \cite[Theorem 3.10]{DDPR2} and Remark \ref{comparing the almost (p,q)-multi-bounds} that $c_{p,q}(B)=1$, so that $B$ is not almost $(p,q)$-multi-bounded whenever $p,q\ge \min\set{r,2}$.

\item[{\rm (iv)}] Consider $E:=\lspace^2\!\textrm{-}\!\bigoplus_{n=1}^\infty\,\lspace^1_n$. Then $B:=E_{[1]}$ is weakly compact. However, it follows from \cite[Example 2.16]{DDPR2} and Remark \ref{comparing the almost (p,q)-multi-bounds}, applying to each component $\lspace^1_n$, that $c_{p,q}(B)\ge 1$ (and thus, $c_{p,q}(B)= 1$), so that $B$ is not almost $(p,q)$-multi-bounded whenever $p,q\in [1,\infty)$.
\end{enumerate}
\end{example}

\begin{problem}
Determine those classes of Banach spaces in which relative weak compactness implies and/or is implied by almost $(p,q)$-multi-boundedness.
\end{problem}

To finish this section, we shall use Theorem \ref{weak compactness in L1 spaces} to give a characterisation of the weak compactness of an  operator from an $\Lfrak^\infty$-space, connecting this property to an asymptotic property of the sequence of $(q,p)$-summing constants of the operator. This characterisation seems to be new even for $\C(K)$-spaces. 

\begin{theorem}\label{weak compactness of operators from abstract M spaces}
Let $E$ be an $\Lfrak^\infty$-space, let $F$ be a Banach space, and let $p, q\in[1,\infty)$. Then an operator $T\in\linearmaps(E,F)$ is weakly compact if and only if 
\begin{align*}\label{weak compactness of operators from abstract M spaces: eq 1}
	\lim_{n\to\infty}\,\frac{\pi_{q,p}^{(n)}(T)}{n^{1/q}}\,=0\,.
\end{align*}
\end{theorem}
\begin{proof}
It is necessary in both directions that $T$ is bounded. In that case, the operator $T$ is weakly compact if and only if $\dual{T}:\dual{F}\to\dual{E}$ is weakly compact, which is the same as saying that the following subset 
\[
	B:=\set{\dual{T}\lambda\colon\ \lambda\in \dual{F}_{[1]}}
\]
of $\dual{E}$ is relatively weakly compact. Since $\dual{E}$ is an $\Lfrak^1$-space, Theorem \ref{weak compactness in L1 spaces} can be applied. So, let $n\in\naturals$. In the following, we denote $\tuple{\lambda}:=(\lambda_1,\ldots,\lambda_n)$ and $\tuple{x}:=(x_1,\ldots,x_n)$. We see that
\begin{align*}
	&\sup\set{\norm{(f_1,\ldots,f_n)}^{(p,q)}_n\colon\ f_1,\ldots,f_n\in B}\\
	&=\quad \sup\set{\norm{(\dual{T}\lambda_1,\ldots,\dual{T}\lambda_n)}^{(p,q)}_n\colon\ \tuple{\lambda}\in(\dual{F}_{[1]})^n}
\end{align*}
\begin{align*}
	&=\quad \sup\set{\left(\sum_{i=1}^n \abs{\duality{x_i}{\dual{T}\lambda_i}}^{\,q}\right)^{1/q}\colon\ \tuple{\lambda}\in(\dual{F}_{[1]})^n\,,\ \tuple{x}\in E^n,\, \mu_{p,n}(\tuple{x})\leq 1}\\
	&=\quad\sup\set{\left(\sum_{i=1}^n \abs{\duality{Tx_i}{\lambda_i}}^{\,q}\right)^{1/q}\colon\ \tuple{\lambda}\in(\dual{F}_{[1]})^n\,,\ \tuple{x}\in E^n,\, \mu_{p,n}(\tuple{x})\leq 1}\\
	&=\quad \sup\set{\left(\sum_{i=1}^n \norm{Tx_i}^{\,q}\right)^{1/q}\colon\ \tuple{x}\in E^n,\, \mu_{p,n}(\tuple{x})\leq 1}=\pi_{q,p}^{(n)}(T)\,.
\end{align*}
Thus, by Theorem \ref{weak compactness in L1 spaces}, the set $B$ is relatively weakly compact if and only if 
\[
	\lim_{n\to\infty}\,\frac{\pi_{q,p}^{(n)}(T)}{n^{1/q}}\,=0\,.
\]
This completes the proof. 
\end{proof}

\section{Combinatorial conditions for amenability} 
\label{Folner type condition section}

\noindent Finally, we are in the position to prove our main theorem. In fact, we shall prove the following strengthening.

\begin{theorem}\label{Folner condition 1}
Let $G$ be a locally compact group. Then the following are equivalent:
\begin{itemize} 
\item[\rm{(i)}] $G$ is amenable;
\item[\rm{(ii)}]  for every $\varepsilon>0$, there exists $n_\varepsilon\in\naturals$ such that for every finite subset $F \subseteq G$ with $\cardinality{F}\ge n_\varepsilon$ there exists a compact subset $C \subseteq G$ with the property that 
\[
	{\haar(EC)}< \varepsilon\cardinality{E}{\haar(C)}\quad\quad(E\subseteq F \ \textrm{with}\ \cardinality{E}\ge n_\varepsilon)\,;
\]

\item[\rm{(iii)}]  there exist $\varepsilon_0\in (0,1)$ and $n_0\in\naturals$ such that for every finite subset $F \subseteq G$ with $\cardinality{F}\ge n_0$ there exists a compact subset $C \subseteq G$ with the property that 
\[
	{\haar(EC)}<\varepsilon_0\cardinality{E}{\haar(C)}\quad\quad(E\subseteq F \ \textrm{with}\ \cardinality{E}\ge n_0)\,.
\]
\end{itemize}
\end{theorem}

\begin{remark}Before giving the proof, let us recall two further known combinatorial conditions, in addition to those mentioned in \S \ref{Introduction}:
\begin{itemize}
	\item[(WF)] There exists $\varepsilon_0\in (0,1)$  
	such that for every finite subset $F\subseteq G$ there exists a compact subset $C\subseteq G$ with the property that
	\[
	{\haar(tC\setminus C)}< \varepsilon_0\,{\haar(C)}\quad\quad(t\in F)\,.
	\]

	\item[(WF*)] There exists $\varepsilon_0\in (0,1)$ such that for arbitrary, finitely many, not necessarily different, elements 	$t_1,\ldots, t_n\in G$ there exists a  compact subset $C\subseteq G$ with the property that
\[
	\frac{1}{n}\sum_{i=1}^n \haar(t_iC\setminus C)< \varepsilon_0\,\haar(C)\,.
\]
\end{itemize}
It is obvious that (F)$\Rightarrow$(WF) and that (WF)$\Rightarrow$(WF*). In the case where $G$ is discrete, F{\o}lner \cite{Folner} proved that (WF*) implies the amenability of $G$. It is possible to modify F{\o}lner's argument to prove this implication for general locally compact groups. In any case, it will follow from our main result of this section that (WF) implies the amenability of $G$, as we shall see below that (WF) implies immediately condition (iii) of Theorem \ref{Folner condition 1}. On the other hand, we are unaware of any way to deduce (WF*) from condition (iii) other than the obvious one that goes through our theorem, and so, even in the discrete case, our result cannot be reduced to F{\o}lner's.
\end{remark}

\begin{remark} It is well-known that (i) implies (F), and obviously (F) implies (ii) while (ii) implies (iii). However, let us give a quick proof of the fact that (WF) implies (iii): Suppose that (WF) holds with some constant $\varepsilon_0\in (0,1)$. Choose $n_0\in\naturals$ such that  
\[
	\delta_0:=	\varepsilon_0+\frac{1}{n_0}<1\,. 
\]
Let $F$ be any finite subsets of $G$. By (WF),  there exists a compact subset $C\subseteq G$ such that
\[
	\haar(tC\setminus C)< \varepsilon_0\haar(C)\quad(t\in F).
\]
It follows that, for each $E\subseteq F$ with $\cardinality{E}\ge n_0$, we have
\[
	\haar(EC)< \haar(C)+\cardinality{E}\,\cdot\, \varepsilon_0\,\haar(C)\le\delta_0\cardinality{E}\haar(C).
\]
Hence, $G$ satisfies (iii), with $\delta_0$ replacing $\varepsilon_0$.
\end{remark}

Thus it remains for us to prove the implication (iii)$\Rightarrow$(i), and for this we shall need some further preparation. 

First of all, instead of a left-invariant mean on $\Linfty(G)$, we shall look for one on $\UC(G)$, where $\UC(G)$ is the space of \emph{bounded uniformly continuous functions} on $G$, i.e. those functions $\varphi\in\C^b(G)$ such that the mappings
\[
	t\mapsto \delta_t\cdot\varphi\qquad\text{and}\qquad t\mapsto\varphi\cdot\delta_t\;,\qquad G\mapsto \C^b(G)\;,
\]
are continuous. Note that $\UC(G)$ is naturally a unital \cstar-subalgebra of $\Linfty(G)$, and so a \emph{mean} on $\UC(G)$ is a positive linear functional $\Lambda:\UC(G)\to\complexs$ such that $\Lambda(\tuple{1})=1$. Moreover, $\UC(G)$ is also a Banach two-sided $\Measures(G)$-submodule of $\Linfty(G)$, and so a mean $\Lambda$ on $\UC(G)$ is thus defined to be \emph{left-invariant} if 
\[
	\delta_s\cdot\Lambda=\Lambda \qquad (s\in G)\,.
\]
This is similar to what was discussed on page \pageref{actions by measures} for $\Linfty(G)$. The fact that we can work with $\UC(G)$ instead of $\Linfty(G)$ to prove the amenability of $G$ is well-known; see \cite[\S 1.1]{Runde} for example.

An advantage that the additional (uniform) continuity in $\UC(G)$ provides us is that we can approximate the action of a function in $\Lone(G)$, considered as a subspace of $\Measures(G)$, on a function in $\UC(G)$, by actions of $\set{\delta_s\colon s\in G}$. In the following, we shall mean by the support of a function $f\in\Lone(G)$ its support when considered as a measure in $\Measures(G)$. We also denote by $\abs{\cdot}_G$ the uniform norm on $G$.

\begin{lemma}\label{approximate by discrete measure}
Let $\varphi\in\UC(G)$, and $f\in\Lone(G)^+_{[1]}$. Then, for each $\varepsilon>0$, there exist a finite subset $\set{t_{j}}$ of $\support(f)$ and a finite sequence  $(\alpha_{j})$ in $\rationals^+$ with $\sum_{j}\alpha_{j}\le 1$ such that
\[
	\abs{\varphi\cdot f-\sum_{j}\alpha_{j}(\varphi\cdot \delta_{t_{j}})}_G < \varepsilon.
\]
\end{lemma}
\begin{proof}
This is more or less obvious as it is immediate from the definition that
\[
	(\varphi\cdot f)(s)=\int \varphi(ts)f(t)\dd t\,.
\] 
The lemma then follows by first approximating $f$ by function with compact support, and then use the continuity of $t\mapsto \varphi\cdot\delta_t,\ G\to \C^b(G)\,$, since $\varphi\in\UC(G)$ by assumption.
\end{proof}

For convenience, let us make the following notation.

\begin{definition}\label{sets of delta}
For each $r\in\integers^+$, a mean $\Lambda$ on $\UC(G)$ is said to be \emph{$r$-good with associated function} $\delta\mapsto n_\delta, [\varepsilon_0^{r/2},\infty)\to\naturals,$ if for every $\delta\ge \varepsilon_0^{r/2}$ and every $n\ge n_{\delta}$ and every elements $s_1,\ldots,s_n$ of $G$, we have
\[
	\norm{(\delta_{s_1}\cdot \Lambda,\ldots,\delta_{s_n}\cdot \Lambda)}_n^{(1,1)}\leq \delta n.
\]
\end{definition}

For example, an arbitrary mean $\Lambda$ on $\UC(G)$ is $0$-good with associated function $\delta\mapsto 1$ defined on $[1,\infty)$.

\begin{lemma}\label{sets of delta to functions}
If $\Lambda$ is $r$-good with the associated function $\delta\mapsto n_\delta$, if $\delta\ge \varepsilon_0^{r/2}$, if $n\ge n_\delta$, and if $f_1,\ldots,f_n\in \Lone(G)^+_{[1]}$, then
\[
	\norm{(f_1\cdot \Lambda,\ldots,f_n\cdot \Lambda)}_n^{(1,1)}\leq \delta n\,,
\] 
\end{lemma}
\begin{proof} Let $\varphi_1,\ldots,\varphi_n\in\UC(G)$ with $\mu_{1,n}(\varphi_1,\ldots,\varphi_n)\le 1\,$, and let $\tau>0$ be arbitrary. For each $1\le k\le n$, by Lemma \ref{approximate by discrete measure}, there exist a finite subset $\set{t_{kj}}$ of $\support(f_k)$ and a finite sequence  $(\alpha_{kj})$ in $\rationals^+$ with $\sum_{j}\alpha_{kj}\le 1$ such that
\[
	\abs{\varphi_k\cdot f_k-\sum_{j}\alpha_{kj}(\varphi_k\cdot \delta_{t_{kj}})}_G < \frac{\tau}{n}.
\]
So, we see that
\begin{align*}
	\sum_{k=1}^n \abs{\duality{\varphi_k}{f_k\cdot\Lambda}}&=\sum_{k=1}^n \abs{\duality{\varphi_k\cdot f_k}{\Lambda}}\\
	&<\sum_{k=1}^n \abs{\duality{\sum_{j}\alpha_{kj}(\varphi_k\cdot \delta_{t_{kj}})}{\Lambda}}+\tau\\
	&=\sum_{k=1}^n \abs{\duality{\varphi_k}{\sum_{j}\alpha_{kj}\delta_{t_{kj}}\cdot\Lambda}}+\tau\\
	&\le \norm{\left(\sum_{j}\alpha_{1j}\delta_{t_{1j}}\cdot \Lambda,\ldots,\sum_{j}\alpha_{nj}\delta_{t_{nj}}\cdot \Lambda\right)}_n^{(1,1)}+\tau\;.
\end{align*}
Let $N\in\naturals$ be a common denominator of all $\alpha_{kj}$. Then, for each $k$, let $(\Lambda_{kl})_{l=1}^N$ be a listing of $N\alpha_{kj}$ number of $t_{kj}\cdot\Lambda$ with the index $j$ varies over all possible value; if this listing does not exhaust the $N$ slots of the sequence $(\Lambda_{kl})_{l=1}^N$, we fill the rest with $0$. The previous inequality then gives
\[
	\sum_{k=1}^n \abs{\duality{\varphi_k}{f_k\cdot\Lambda}}\le \frac{1}{N}\sum_{l=1}^N\norm{\left(\Lambda_{1l},\ldots,\Lambda_{nl}\right)}_n^{(1,1)}+\tau\le \delta n+\tau;
\]
the last inequality follows from the assumption on $\Lambda$. Since $\tau$ was arbitrary, it follows that $\norm{(f_1\cdot \Lambda,\ldots,f_n\cdot \Lambda)}_n^{(1,1)}\leq \delta n$.
\end{proof}

Before continuing, we shall need a technical lemma of a combinatorial nature. The motivation for this is roughly that, says in the simpler case of discrete groups, we are going to have two finite subsets $E$ and $S$ of a group $G$, and each element $t\in E$ associates with a function $\phi_t$ of uniform norm one. Suppose that these functions have pairwise disjoint supports. In our calculation later, for each $u\in ES$, we will have to add all $\phi_t$ where $u\in tS$ together to obtain a function $\psi_u$. This $\psi_u$ is still of norm one. But we also have to have to add different $\psi_u$ together while keeping the uniform norm one. This corresponds to collecting different $u$ together but only allow to put two $u$ in the same collection if they, as elements of $ES$, do not have the same $E$-divisor. For estimation purpose, we also want to keep the number of collections as small as possible. Obviously, the number of collections must be at least $\cardinality{S}$, but this is not always achievable as the two subsets $E=\set{0,1,2}$ and $S=\set{0,1}$ of $\integers_3$ show. However, the following lemma will be sufficient for our purpose. 

\begin{lemma}\label{a combinatorial rearrangement}
Let $R_1,\ldots, R_n$ be pairwise disjoint finite sets, each of cardinality at most $m$. For each $x\in R_i$, set $\rho(x):=i$. Let $\set{A_u\colon u\in \Tcal}$ be a partition of $\bigcup_iR_i$. Suppose that $\rho$ is injective on each $A_u$. Then there is a partition $\set{\Dcal_1,\ldots,\Dcal_K,\Ecal}$ of $\Tcal$ such that
\begin{enumerate}
	\item[{\rm (a)}] $\rho$ is injective on the union $\bigcup_{u\in\Dcal_k}A_u$, for each $1\le k\le K$, 
	\item[{\rm (b)}] $K\le n^{3/4}m$, and $\cardinality{\Ecal}\le 2n^{3/4}m$.
\end{enumerate}
\end{lemma}

If $E$ and $S$ are subsets of a group $G$, says $E=\set{t_1,\ldots, t_n}$, then the sets $R_i$ and $A_u$ that are relevant to the discussion that precedes the statement of the lemma are $R_i:=\set{(t_i,s)\colon\ s\in S}$ and $A_u:=\set{(t,s)\in E\times S\colon ts=u}$ for $u\in \Tcal:=ES$. In fact, it will be helpful in the following proof to view each $R_i$ as row number $i$, and $\rho(x)$ will tell which row the element $x$ belongs. The condition that $\rho$ is injective on some set $A$ then just means that no two elements of $A$ are on the same row, etc.

\begin{proof}
Since the number of $u\in\Tcal$ such that $\cardinality{A_u}\ge n^{1/4}$ is at most $n^{3/4}m$, we are done if the problem can be solved under the additional assumption that  
\begin{align}\label{a combinatorial rearrangement: eq -1}
	\cardinality{A_u}<n^{1/4}\quad\text{for each $u$}\,,
\end{align}
and with a stronger conclusion that
\begin{enumerate}
	\item[{\rm (b')}] $K\le n^{3/4}m$, and $\cardinality{\Ecal}\le n^{3/4}m$
\end{enumerate}
in addition to (a). Thus for the rest of the proof we shall suppose that \eqref{a combinatorial rearrangement: eq -1} holds.

The proof in this case will be carried out through the induction on the bound $m$ of the cardinalities of $R_i$. If $m=1$, then we can take $K:=1$ and $\Dcal_1:=\Tcal$, and $\Ecal:=\emptyset$. Suppose now that $m>1$. By adding more elements to the sets $R_i$ and corresponding singleton sets into the original collection $\set{A_u\colon u\in\Tcal}$ if necessary, we shall suppose that 
\begin{align}\label{a combinatorial rearrangement: eq 0}
	\text{each ``row'' $R_i$ has exactly $m$ elements.} 
\end{align}

Let us call in this proof a sequence $(\Dcal_1, \ldots, \Dcal_l)$ of subsets of $\Tcal$ \emph{admissible} if
\begin{enumerate}
	\item $\rho$ is injective on $\bigcup_{u\in\Dcal_i}A_u$ for each $1\le i\le l$, and
	\item  for each $1\le i\le l$
	\[
		\cardinality{\bigcup_{u\in \Dcal_{i}}\rho(A_u)\setminus \bigcup_{j=1}^{i-1}\bigcup_{u\in \Dcal_j}\rho(A_u)}\ge  n^{1/4}\,.
	\]
\end{enumerate}
Note that, with the usual convention, we allow $l=0$ for an empty sequence.

\underline{Claim:} \emph{If  $(\Dcal_1, \ldots, \Dcal_l)$ is an admissible sequence of subsets of $\Tcal$ and 
	\begin{align}\label{a combinatorial rearrangement: eq 1}
		\cardinality{\bigcup_{j=1}^{l}\bigcup_{u\in \Dcal_j}\rho(A_u)}<n-n^{3/4}
	\end{align}
then there is a subset $\Dcal_{l+1}$ of $\Tcal$ such that $(\Dcal_1, \ldots, \Dcal_l, \Dcal_{l+1})$ is still admissible.} Indeed, let us consider the set $\Fcal$ of all $v\in\Tcal$ such that $\rho(A_v)\nsubseteq \bigcup_{j=1}^{l}\bigcup_{u\in \Dcal_j}\rho(A_u)$. Notice that if $v_1,\ldots, v_k\in \Tcal$ are such that
\begin{align}\label{a combinatorial rearrangement: eq 3}
	\cardinality{\bigcup_{i=1}^k\rho(A_{v_i})}<n^{1/2}\,,
\end{align}
then there must be an $v_{k+1}\in \Fcal$ such that $\rho(A_{v_{k+1}})\cap\bigcup_{i=1}^k\rho(A_{v_i})=\emptyset$, i.e. $A_{v_{k+1}}$ contains only elements that are not in the same row as any one from $\bigcup_{i=1}^k A_{v_i}$. For otherwise, for every $v\in \Fcal$, $A_v$ must have an element in one of the row numbered in $\bigcup_{i=1}^k\rho(A_{v_i})$, and so
\[
	\cardinality{\Fcal}\le \cardinality{\bigcup_{i=1}^k\rho(A_{v_i})}\cdot m<n^{1/2}m\,.
\]
Thus since each $A_v$ is assumed (from the first paragraph) to have at most $n^{1/4}$ elements, the cardinality of $\bigcup_{v\in\Fcal}A_v$ is at most $\cardinality{\Fcal}\cdot n^{1/4}<n^{3/4}m$. This, the definition of $\Fcal$, and the assumption \eqref{a combinatorial rearrangement: eq 0}, then imply that there are $<n^{3/4}$ rows that are not numbered in $\bigcup_{j=1}^{l}\bigcup_{v\in \Dcal_j}\rho(A_v)$. Combining with \eqref{a combinatorial rearrangement: eq 1}, we obtain that there are $<n$ rows; a contradiction. Thus $v_{k+1}$ with the above specified property can be found whenever \eqref{a combinatorial rearrangement: eq 3} holds. This can then be used to inductively construct (elements of) a finite subset $\Dcal$ of $\Fcal$ such that $\rho(A_v)\cap\rho(A_{v'})=\emptyset$ whenever $v\neq v'\in \Dcal$ and until we obtain
\begin{align}\label{a combinatorial rearrangement: eq 3.5}
	\cardinality{\bigcup_{v\in \Dcal}\rho(A_v)}\ge n^{1/2}\,.
\end{align}
Take this $\Dcal$ as our $\Dcal_{k+1}$. Then condition (i) for the admissibility of $(\Dcal_1,\ldots, \Dcal_l, \Dcal_{l+1})$ is readily seen to be satisfied. For (ii): by \eqref{a combinatorial rearrangement: eq -1} and \eqref{a combinatorial rearrangement: eq 3.5}, we obtain  $\cardinality{\Dcal_{l+1}}\ge n^{1/4}$, and since $\Dcal_{l+1}\subseteq \Fcal$ and $\rho(A_v)\cap\rho(A_{v'})=\emptyset$ whenever $v\neq v'\in \Dcal_{l+1}$, we see that 
\[
	\cardinality{\bigcup_{v\in \Dcal_{l+1}}\rho(A_v)\setminus \bigcup_{j=1}^{l}\bigcup_{u\in \Dcal_j}\rho(A_u)}\ge \abs{\Dcal_{l+1}}\ge n^{1/4}\,.
\]
This completes the proof of Claim.

Take a maximal admissible sequence $(\Dcal_1, \ldots, \Dcal_l)$ of subsets of $\Tcal$, i.e. an admissible sequence that is not a proper initial segment of any other one. Then by the above claim we must have
\begin{align*}
		\cardinality{\bigcup_{j=1}^{l}\bigcup_{u\in \Dcal_j}\rho(A_u)}\ge n-n^{3/4}\,.
\end{align*}
This means that there exists a subset $\Ecal_0$ of $\Tcal$ with 
\begin{align}\label{a combinatorial rearrangement: eq 5}
	\cardinality{\Ecal_0}\le n^{3/4}\,,
\end{align}
such that each row not numbered in $\bigcup_{j=1}^{l}\bigcup_{u\in \Dcal_j}\rho(A_u)$ intersects an $A_v$ for some $v\in\Ecal_0$. Thus if we set 
\[
	R'_i:=R_i\setminus \bigcup_{u\in \bigcup_{i=1}^l\Dcal_i\cup \Ecal_0}A_u\,,
\]
then each $R'_i$ has cardinality at most $m-1$, and the collection 
\[
	\Tcal':=\Tcal\setminus\left(\bigcup_{i=1}^l\Dcal_i\cup \Ecal_0\right)
\]
is a partition of $\bigcup_{i=1}^n R'_i$. Obviously, condition \eqref{a combinatorial rearrangement: eq -1} still holds. Thus, by the inductive hypothesis, there is a partition $\set{\Dcal_{l+1},\ldots,\Dcal_K, \Ecal'}$ of $\Tcal'$ such that
\begin{enumerate}
	\item $\rho$ is injective on the union $\bigcup_{u\in\Dcal_k}A_u$, for each $l+1\le k\le K$, 
	\item $K-l\le n^{3/4}(m-1)$, and $\cardinality{\Ecal'}\le n^{3/4}(m-1)$.
\end{enumerate}
Set $\Ecal:=\Ecal'\cup\Ecal_0$. We \emph{claim} that $\set{\Dcal_1,\ldots,\Dcal_l, \Dcal_{l+1},\ldots,\Dcal_K, \Ecal}$ is a partition of $\Tcal$ that satisfies (a) and (b'). Indeed, condition (a) is the totality of (i) above and condition (i) for the admissibility of $\set{\Dcal_1,\ldots,\Dcal_l}$. For (b'), we first note that the fact that $\cardinality{\Ecal}\le n^{3/4}m$ follows from (ii) above and \eqref{a combinatorial rearrangement: eq 5}. Finally, by condition (ii) for the admissibility of $\set{\Dcal_1,\ldots,\Dcal_l}$, we see that
\[
	n\ge \cardinality{\bigcup_{j=1}^{l}\bigcup_{u\in \Dcal_j}\rho(A_u)}\ge l\cdot n^{1/4}\,.
\]
Hence $l\le n^{3/4}$, and so $K\le n^{3/4}(m-1)+l\le n^{3/4}m$.

Thus, the lemma is proved.
\end{proof}

Let us return to our proof of Theorem \ref{Folner condition 1}. We remark that in our calculation of $\norm{\cdot}_n^{(1,1)}$ on $\dual{\UC(G)}$ below, we use the formula that
\[
	\mu_{1,n}(\varphi_1,\ldots,\varphi_n)=\abs{\sum_{i=1}^n\abs{\varphi_i}}_G\qquad(\varphi_i\in\UC(G))\,.
\]

\begin{lemma} \label{lowering the epsilon}
Under condition {\rm (iii)} of Theorem \ref{Folner condition 1}, suppose that $\Lambda_r$ is $r$-good with  associated function $\delta\mapsto n_\delta$. Then there exists a $\Lambda_{r+1}$ that is $(r+1)$-good whose associated function extends that for $\Lambda_r$.
\end{lemma}
\begin{proof} 
For each $\gamma\in [\varepsilon_0^{(r+1)/2},\varepsilon_0^{r/2})$, write $\gamma=\varepsilon_0^{1/2}\delta$ where $\delta\ge \varepsilon_0^{r/2}$, and then choose $n_{\gamma}\in\naturals$ such that 
\begin{align}\label{lowering the epsilon: eq -1}	
	 n_\gamma\ge n_\delta\quad\text{and}\quad\varepsilon_0\delta+\frac{3n_\delta}{n_\gamma^{1/4}}<\gamma\,.
\end{align}
This extends the domain of $\delta\mapsto n_\delta$ from $[\varepsilon_0^{r/2},\infty)$ to $[\varepsilon_0^{(r+1)/2}, \infty)$.

Let $F$ be a finite subset of $G$. By (iii), we can find a compact subset $C \subseteq G$ such that 
\begin{align}\label{lowering the epsilon: eq 0}	
	{\haar(EC)}< \varepsilon_0\cardinality{E}{\haar(C)}\quad\quad(E\subseteq F \ \textrm{with}\ \cardinality{E}\ge n_0).
\end{align}
Set $\Psi_F=(\chi_C\cdot\Lambda_r)/{\haar(C)}$. Then $\Psi_F$ is a mean on $\UC(G)$. We \emph{claim} that \begin{align}\label{lowering the epsilon: eq 4}	
	\norm{(\delta_{t_1}\cdot \Psi_F,\ldots,\delta_{t_n}\cdot \Psi_F)}_n^{(1,1)}\leq \gamma n
\end{align}
for every $\gamma\ge \varepsilon_0^{(r+1)/2}$, every $n\ge n_{\gamma}$ and every distinct $t_1,\ldots,t_n$ in $F$. 

Indeed, let $E:=\set{t_1,\ldots,t_n}\subseteq F$ with $n=\cardinality{E}\ge n_{\gamma}$, and then take $\varphi_1,\ldots,\varphi_n\in\UC(G)^+$ with $\sum_{l=1}^n\varphi_l\le 1$. We see that 
\begin{equation}
\begin{aligned}\label{lowering the epsilon: eq 1}	
\sum_{l=1}^n{(\delta_{t_l}\cdot\Psi_F)(\varphi_l)}&=\frac{1}{\haar(C)}\sum_{l=1}^n{((\delta_{t_l}*\chi_{C})\cdot\Lambda_r)(\varphi_l)}\\
&=\frac{1}{\haar(C)}\sum_{l=1}^n{(\chi_{t_lC}\cdot\Lambda_r)(\varphi_l)}. 
\end{aligned}
\end{equation}
And so \eqref{lowering the epsilon: eq 4} holds if $\gamma\ge \varepsilon_0^{r/2}$ by Lemma \ref{sets of delta to functions}. Thus we may (and shall) suppose that $\gamma\in [\varepsilon_0^{(r+1)/2},\varepsilon_0^{r/2})$. Then $\gamma=\varepsilon_0^{1/2}\delta$ and \eqref{lowering the epsilon: eq -1} holds. Let $\tau>0$ be arbitrary. We see that there exist $m\in\naturals$ and, for each $l\in\set{1, \ldots, n}$, a sequence $(C_{lj})_{j=1}^m$ of pairwise disjoint, measurable subsets of $t_lC$ satisfying the following conditions:
\begin{enumerate}
	\item[{\rm (i)}] all the $C_{lj}$ have the same non-zero measure $\mu_0$;
	\item[{\rm (ii)}] $\haar(t_lC\setminus\bigcup_{j=1}^m C_{lj})\le \tau \haar(C)/n$\,\; ($1\le l\le n$); and
	\item[{\rm (iii)}] for every $(k,i)$ and $(l,j)$, either $C_{ki}=C_{lj}$ or $C_{ki}\cap C_{lj}=\emptyset$.
\end{enumerate} 
This is possible because, if $G$ is discrete, then we can even let $\tau=0$ by choosing $C_{lj}$ to be singletons, whereas in the case where $G$ is non-discrete, the Haar measure is non-atomic, and so we can start by finding partitions of the $t_lC$ that satisfy (iii), and then refine them further to have in addition (i), with condition (ii) is to collect any ``leftover''. Set $C'_l:=t_lC\setminus\bigcup_{j=1}^m C_{lj}$. We see from \eqref{lowering the epsilon: eq 1} that
\begin{equation}
\begin{aligned}\label{lowering the epsilon: eq 2}	
\sum_{l=1}^n{(\delta_{t_l}\cdot\Psi_F)(\varphi_l)}&=\frac{1}{\haar(C)}\sum_{l=1}^{n}\sum_{j=1}^{m}(\chi_{C_{lj}}\cdot\Lambda_r)(\varphi_{l})+\frac{1}{\haar(C)}\sum_{l=1}^{n}(\chi_{C'_l}\cdot\Lambda_r)(\varphi_{l})\\
&\le\frac{1}{\haar(C)}\sum_{l=1}^{n}\sum_{j=1}^{m}(\chi_{C_{lj}}\cdot\Lambda_r)(\varphi_{l})+\tau\;.
\end{aligned}
\end{equation}
Set $R_l:=\set{(l,j)\colon 1\le j\le m}$ for each $1\le l\le n$. Then set 
\[
	\Tcal:=\set{\chi_{C_{lj}}\colon\ 1\le l\le n,\ 1\le j\le m}
\] 
and, for each $u\in \Tcal$, define $A_u:=\set{(l,j)\colon \chi_{C_{lj}}=u}$. From (i) and (iii) above, we see that
\begin{align}\label{lowering the epsilon: eq 3}	
	m\mu_0\le \haar(C)\qquad\text{and}\qquad \abs{\Tcal}\mu_0\le \haar(EC)\,.
\end{align}
By Lemma \ref{a combinatorial rearrangement}, there is a partition $\set{\Dcal_1,\ldots,\Dcal_K}$ of $\Tcal$ such that
\begin{enumerate}
	\item[{\rm (a)}] for each $1\le k\le K$, no two elements of $\bigcup_{u\in\Dcal_k}A_u$ are on the same row $R_l$ for any $l$, 
	\item[{\rm (b)}] $K\le 3n^{3/4}m$ (where we simply break $\Ecal$ provided by the lemma into a collection of singletons, and add them to the collection $\set{\Dcal_k\colon k\in K}$). 
\end{enumerate}
By (a), if for each $u\in\Tcal$, we set
\[
	\psi_u:=\sum_{(l,j)\in A_u}\varphi_l,
\]
then $\psi_u\in\UC(G)^+$ and $\sum_{u\in\Dcal_k}\psi_u\le 1$. Thus, if $q_k:=\abs{\Dcal_k}\ge n_{\delta}$, then, by Lemma \ref{sets of delta to functions}, we have 
\[
	\sum_{u\in\Dcal_k}(u\cdot\Lambda_r)(\psi_u)\le \norm{(u_1\cdot\Lambda_r,\ldots,u_{q_k}\cdot\Lambda_r)}^{(1,1)}<\delta q_k\max_{1\le j\le q_k}\norm{u_j}_{\Lone(G)}=\delta \cardinality{\Dcal_k}\mu_0\,,
\]
where $u_1,\ldots,u_{q_k}$ is any (temporal) listing of $\Dcal_k$. On the other hand, if $\abs{\Dcal_k}<n_{\delta}$, then 
\[
	\sum_{u\in\Dcal_k}(u\cdot\Lambda_r)(\psi_u)\le\sum_{u\in\Dcal_k}\norm{u}_{\Lone(G)}\le n_{\delta}\mu_0.
\]
Putting the above two inequalities into \eqref{lowering the epsilon: eq 2}, we see that
\begin{align*}
\sum_{l=1}^n{(\delta_{t_l}\cdot\Psi_F)(\varphi_l)}&\le\frac{1}{\haar(C)} \sum_{u\in\Tcal} (u\cdot \Lambda_r)(\psi_u)+\tau\\
&=\frac{1}{\haar(C)} \sum_{k=1}^K\sum_{u\in\Dcal_k} (u\cdot \Lambda_r)(\psi_u) +\tau\\
&\le\frac{1}{\haar(C)} \sum_{k=1}^K\left(\delta \cardinality{\Dcal_k}\mu_0+n_{\delta}\mu_0\right)+\tau\\
&\le\frac{1}{\haar(C)} \left(\delta \cardinality{\Tcal}\mu_0+n_{\delta} K\mu_0\right)+\tau\\
&\le\frac{1}{\haar(C)} \left(\delta \haar(EC)+n_{\delta} 3n^{3/4}\haar(C)\right)+\tau\\
&= n\left(\delta \frac{\haar(EC)}{\abs{E}\haar(C)}+\frac{3n_{\delta}}{n^{1/4}}\right)+\tau\le n\gamma+\tau\,;
\end{align*}
where for the fifth line, we use \eqref{lowering the epsilon: eq 3} and the defining condition (b) for $K$, while for the last inequality, we use \eqref{lowering the epsilon: eq -1} and \eqref{lowering the epsilon: eq 0} and the assumption that $n=\abs{E}\ge n_{\gamma}$. Since $\tau>0$ is arbitrary, we obtain \eqref{lowering the epsilon: eq 4}. This completes the proof of the \emph{claim}.

To finish the proof of the lemma, make the collection of all finite subsets of $G$ into a net as usual, and choose $\Lambda_{r+1}$ to be a weak-$*$ cluster point of $\Psi_F$ as $F$ runs along this net.  Then \eqref{lowering the epsilon: eq 4} shows that $\Lambda_{r+1}$ is $(r+1)$-good with associated function extending that of $\Lambda_r$.
\end{proof}

Finally, we can now complete the proof of the implication (iii)$\Rightarrow$(i) of Theorem \ref{Folner condition 1} (and thereby complete the proof of the theorem).

\begin{proof}[Proof of Theorem \ref{Folner condition 1}]
Suppose that condition (iii) of Theorem \ref{Folner condition 1} holds. Take any mean $\Lambda_0$ on $\UC(G)$, which is $0$-good with associated function $\delta\mapsto 1$ on $[1,\infty)$. By Lemma \ref{lowering the epsilon}, we see that there exists a function $\delta\mapsto n_\delta$ on $(0,\infty)$ and, for each $r\in\naturals$, a mean $\Lambda_r$ on $\UC(G)$ such that $\Lambda_r$ is $r$-good with associated function the restriction of $\delta\mapsto n_\delta$. Take $\Lambda$ to be a weak-$*$ cluster point of the $\Lambda_r$ as $r\to \infty$. Then $\Lambda$ is obviously a mean on $\UC(G)$. Moreover, for every $\delta>0$ and every $s_1,\ldots,s_n$ of $G$ with $n\ge n_\delta$, we have
\[
	\norm{(\delta_{s_1}\cdot \Lambda_r,\ldots,\delta_{s_n}\cdot \Lambda_r)}_n^{(1,1)}\leq \delta n
\]
for all $r\ge 2\ln\delta/\ln\varepsilon_0$, and so
\begin{align*}
	\norm{(\delta_{s_1}\cdot \Lambda,\ldots,\delta_{s_n}\cdot \Lambda)}_n^{(1,1)}\leq \delta n.
\end{align*}
Thus, $\set{\delta_s\cdot\Lambda\colon s\in G}$ is almost $(1,1)$-multi-bounded, and so its closed convex hull $\Ccal$ is weakly compact by Theorem \ref{weak compactness in Lone when multi-norm grow slowly} and the Kre\u{\i}n--\u{S}mulian theorem. For each $s\in G$, consider the map
\[
	L_s:\Psi\mapsto \delta_s\cdot\Psi\,,\quad \Ccal\to \Ccal\,.
\]
 We obtain a group $\set{L_s\colon s\in G}$ of isometric affine maps. By the Ryll-Nardzewski fixed point theorem \cite{R-N}, there exists $\tilde{\Lambda}\in \Ccal$ which is a common fixed point for all $L_s$ ($s\in G$). Obviously, $\tilde{\Lambda}$ must be a left-invariant mean on $\UC(G)$. Hence $G$ is amenable. 
\end{proof}

In an attempt to resolve the question of when $\Lspace^p(G)$ is injective, Dales and Polyakov \cite[Definition 5.5]{DP} introduced the notion of \emph{pseudo-amenability} (for discrete groups) as follows. A locally compact group $G$ is said to be \emph{pseudo-amenable} if it satisfies the following condition:
\begin{itemize}
\item[(PA)] For every $\varepsilon >0$, there exists $n_\varepsilon \in \naturals$ such that, for every finite subset $F\subseteq G$ with $\cardinality{F}\ge n_\varepsilon$, there exists a compact subset $C \subseteq G$ such that
\[
{\haar(FC)}< \varepsilon \cardinality{F}{\haar(C)}\,.
\]
\end{itemize}
It can be seen that amenability implies pseudo-amenability, for example, since the condition (ii) in Theorem \ref{Folner condition 1} obviously implies (PA). On the other hand, it is proved in \cite{DP} that a pseudo-amenable discrete group cannot contain the free group of two generators. However, it is unknown if a pseudo-amenable group is necessarily amenable. A non-amenable pseudo-amenable group if exists must be a non-amenable group without a free subgroup of two generators. The first such group was constructed in \cite{O}, which is very difficult even to describe, and only recently that simple examples of such groups was provided in \cite{Monod}. However, we do not know whether or not these groups are pseudo-amenable.

In the light of Theorem \ref{Folner condition 1}, let us introduce the following weaker version of (PA) for a locally compact group $G$:
\begin{itemize}
\item[(WPA)] there exists $\varepsilon_0\in(0,1)$ and $n_0 \in \naturals$ such that, for every finite subset $F\subseteq G$ with $\cardinality{F}\ge n_0$, there exists a compact subset $C \subseteq G$ such that
\[
	{\haar(FC)}< \varepsilon_0 \cardinality{F}{\haar(C)}\,.
\]
\end{itemize}
As expected, a discrete group satisfying (WPA) cannot contain the free group of two generators. 

\begin{proposition}
Suppose that a group $G$ contains the free group on two generators. Then $G$, with the discrete topology, does not satisfy \emph{(WPA)}.
\end{proposition}
\begin{proof}
It follows that, for every $n\in\naturals$, $G$ contains a free group on $n$ generators. Fix $n\in\naturals$, and suppose that $s_1,\ldots, s_n$ are the generators of a free subgroup of $G$. The result will follow at once if we can prove that, for every finite subset $C\subseteq G$, we have
\begin{align*}
	\cardinality{FC}> (n-1)\cardinality{C}\,,
\end{align*}
where $F=\set{s_1,\ldots, s_n}$.  In fact, we \emph{claim} that for every finite subset $C\subseteq G$,
\begin{align*}
	\cardinality{FC\setminus C}> (n-1)\cardinality{C}.
\end{align*}
This is a stronger conclusion than the  `only if' part of \cite[Theorem 3]{KromKrom}; the proof is, however, by induction as is the proof of the latter.
\end{proof}

\begin{problem}
Let $G$ be a locally compact group. Suppose that $G$ satisfies (WPA) or (PA). Is $G$ necessarily amenable?
\end{problem}

\section*{}

\vspace{-0.5cm}

\subsection*{Acknowledgements}
This work grew out of my involvement with \cite{DDPR1} and \cite{DDPR2}, and I would like to thank my co-authors of these two papers, H. Garth Dales (Lancaster), Matthew Daws (Leeds), and Paul Ramsden (Leeds) for the fruitful collaboration and stimulating discussions during these projects. The author would also like to thank the anonymous referee for their comment that have improved the presentation of the paper.

\end{document}